\documentclass[
11pt,                          
english                        
]{article}

\usepackage[english]{babel}    
\usepackage{amsmath}           
\usepackage[utf8]{inputenc}    
\usepackage[T1]{fontenc}       
\usepackage{exscale}           
\usepackage[sort]{cite}        
\usepackage[a4paper]{geometry} 
\usepackage{xspace}            
\usepackage{amssymb}           
\usepackage{eucal}               
\usepackage{bbm}               
\usepackage{mathtools}         
\usepackage[amsmath,thmmarks,hyperref]{ntheorem} 
\usepackage{mathrsfs}          
\usepackage{stmaryrd}          
\usepackage{paralist}          
\usepackage[nottoc]{tocbibind} 
\usepackage[backref=page,      
           pdfpagelabels       
          ]{hyperref}         

\usepackage{color}

\title{Quantization of Hamiltonian coactions via twist}

\author{
  \textbf{Pierre Bieliavsky}\thanks{\texttt{pierre.bieliavsky@uclouvain.be}}\\[0.3cm]
  Faculté des Sciences\\
  Ecole de Mathématique (MATH)\\
  Institut de Recherche en Mathématique et Physique (IRMP)\\
  Chemin du Cyclotron 2 bte L7.01.02 \\
  1348 Louvain-la-Neuve\\
  Belgium\\[0.3cm]
  \textbf{Chiara Esposito}\thanks{\texttt{chiara.esposito@mathematik.uni-wuerzburg.de}},\\
    Institut für Mathematik \\
  Universität Würzburg \\
  Campus Hubland Nord \\
  Emil-Fischer-Straße 31 \\
  97074 Würzburg \\
  Germany \\[0.3cm]
  \textbf{Ryszard Nest}\thanks{\texttt{rnest@math.ku.dk}},\\
  Department of Mathematical Sciences \\
  Universitetsparken 5 \\
  DK-2100 Copenhagen Ø \\ 
  Denmark
   \\[0.3cm]
}

\newcommand{\refitem}[1] {\textit{\ref{#1}.)}}

\numberwithin{equation}{section}

\allowdisplaybreaks

\let\originalleft\left
\let\originalright\right
\renewcommand{\left}{\mathopen{}\mathclose\bgroup\originalleft}
\renewcommand{\right}{\aftergroup\egroup\originalright}

\theoremheaderfont{\normalfont\bfseries}
\theorembodyfont{\itshape}
\newtheorem{lemma}{Lemma}[section]
\newtheorem{proposition}[lemma]{Proposition}
\newtheorem{theorem}[lemma]{Theorem}

\newtheorem{definition}[lemma]{Definition}

\theorembodyfont{\rmfamily}

\newtheorem{example}[lemma]{Example}
\newtheorem{remark}[lemma]{Remark}

\theoremheaderfont{\scshape}
\theorembodyfont{\normalfont}
\theoremstyle{nonumberplain}
\theoremseparator{:}
\theoremsymbol{\hbox{$\boxempty$}}
\newtheorem{proof}{Proof}

 \newenvironment{remarklist}{\begin{compactenum}[\itshape i.)]}{\end{compactenum}}
 \newenvironment{lemmalist}{\begin{compactenum}[\itshape i.)]}{\end{compactenum}}
 \newenvironment{theoremlist}{\begin{compactenum}[\itshape i.)]}{\end{compactenum}}
 
 \newenvironment{definitionlist}{\begin{compactenum}[\itshape i.)]}{\end{compactenum}}

\newcommand{\bem}[1]        {\,{}^\flat{\hspace*{-0.05cm}#1}}

 \newcommand\ot[2]{\overset{#1}{#2}}
\newcommand{\lforms}[1]    {\Omega^1_{\mathrm{loc}}(#1)}
\newcommand{\Fun}[1][k]      {\mathscr{C}^{#1}}
\newcommand{\Cinfty}         {\Fun[\infty]}
\newcommand{\Lie}   {\mathscr{L}}
\newcommand{\lie}[1]          {\mathfrak{#1}}
\newcommand{\tensor}[1][{}]           {\mathbin{\otimes_{\scriptscriptstyle{#1}}}}
\DeclarePairedDelimiter{\Schouten}{\llbracket}{\rrbracket}
\newcommand{\Sec}[1][k]      {\Gamma^{#1}}
\newcommand{\Secinfty}       {\Sec[\infty]}
\newcommand{\argument}       {\,\cdot\,}
\newcommand{\D}              {\mathop{}\!\mathrm{d}}
\newcommand{\group}[1]        {\mathrm{#1}}
\newcommand{\Der}            {\operatorname{\mathsf{Der}}}
\newcommand{\at}[1]          {\big|_{#1}}

\DeclareMathOperator{\id}    {\mathsf{id}}
\newcommand{\pr}             {\mathrm{pr}}
\DeclareMathOperator{\ad}     {\mathrm{ad}}

\begin{document}

\maketitle

\begin{abstract}
  In this paper we introduce a notion of quantum Hamiltonian
  (co)action of Hopf algebras endowed with Drinfel'd twist structure (resp., 2-cocycles). First, we define
  a classical Hamiltonian action in the setting of Poisson Lie groups compatible with the 2-cocycle
  structure and we discuss a concrete example. This allows us to construct, out of the classical momentum
  map, a quantum momentum map in the setting of Hopf coactions and to quantize
  it by using Drinfel'd approach.
\end{abstract}

\newpage

\tableofcontents

\section*{Introduction}

Deformation quantization has been introduced by Bayen,
Flato, Fronsdal, Lichnerowicz and Sternheimer in
\cite{bayen.et.al:1978a} and since then many developments occurred.  
A (formal) star product on a Poisson manifold $M$ is defined as a
formal associative deformation of the algebra of smooth
functions $\Cinfty (M)$ on $M$. Existence and classification
of star products on Poisson manifolds has been proved via formality theory
in \cite{kontsevich}.
In the same spirit, Drinfel'd introduced the notion
of quantum groups as deformations of Hopf algebras, 
whose semiclassical limit are the so-called Poisson Lie groups
which are Lie groups with multiplicative Poisson structures
(see
e.g. the textbooks \cite{chari.pressley:1994a,
Majid} for a detailed discussion).

In this paper we focus on 
particular classes of star products which are induced by a (formal) Drinfel'd twist by means
of universal deformation formulas (UDF) as discussed e.g. in \cite{drinfeld:1983, drinfeld:1988a}. 
Roughly speaking, a Drinfel'd
twist of an enveloping algebra $\mathcal{U}(\lie g)$ is an element $\mathcal{F} \in \mathcal{U}(\lie{g}) \tensor
\mathcal{U}(\lie{g})$ compatible with the Hopf algebra structure on $\mathcal{U}(\lie{g})$.
Given a Hopf algebra action of $\mathcal{U}(\lie{g})$ on an associative algebra
one can deform the $\mathcal{U}(\lie{g})$-module algebra and the deformed
product turns out to
be a star product. It is important to stress that the $\mathcal{U}(\lie{g})$-module algebra
is automatically endowed with a Poisson bracket defined as the semiclassical
limit of such star product.
In recent works the UDF has been further studied, e.g. \cite{BBM, jonas, thomas1, thomas2, gerstenhaber}.
Also, a twist defines a 2-cocycle on the Hopf algebra $\Cinfty(G)$ and it can be seen 
that the star products induced via UDF coincide with star products
induced by the 2-cocycle on $\Cinfty(G)$-comodule algebras. Finally, a non-formal 
version of Drinfel'd twist and its corresponding UDF has been discussed in \cite{Bieliavsky2015}.

Given a Lie algebra action $\varphi\colon \lie g \to \Secinfty (TM)$ on a smooth manifold $M$, we can always
obtain a Hopf algebra action $\mathcal{U}({\lie{g}}) \times \Cinfty (M) \to \Cinfty (M)$. Thus,
Drinfel'd approach can be interpreted by saying that symmetries encoded by Lie algebra
actions induce quantization. Also, this approach provides a notion of
quantized action. In this paper we prove that 
this approach is compatible with Hamiltonian actions. 
In other words, given a classical Hamiltonian action our goal is to quantize it
by using Drinfel'd approach and get a notion of quantum momentum map. The problem of quantizing 
the momentum map has been the main topic of
many works, e.g. \cite{fedosov} and \cite{lu3}. In general, the interest for the quantization
of the momentum map is motivated by the fact that conserved quantities described via the momentum map lead
to phase space reduction which constructs from the high-dimensional original phase space one of a
smaller dimension. Thus, it is highly desirable to
find an analogue in the quantum setting.
A study of the compatibility of the notion of quantum action 
provided by Drinfeld and the notion of Hamiltonian action was so far absent.
In this paper we prove that the two notions are actually compatible and we construct 
a quantum momentum map via twist.

The content of this work is as follows. 

In Section~\ref{sec:preliminaries}
we discuss the well-known notions of Drinfel'd twist and its corresponding
2-cocycle and the construction of the universal deformation
formula. Twist and 2-cocycle induce a quantum group 
structure which is briefly recalled.

Section~\ref{sec:hamiltonian} contains a definition of
Hamiltonian actions in the setting of Poisson Lie groups which
generalizes the one contained in \cite{Lu1990, Lu1991}. 
More precisely, we need to introduce a notion of classical Hamiltonian action
which is compatible with twist, which is necessary in order to quantize Hamiltonian
actions by using Drinfel'd approach. 

It is known that the semiclassical limit of a twist
gives rise to an element $r \in \lie g\wedge \lie g$,
called $r$-matrix, satisfying the condition $\Schouten{r, r}= 0$ (for a detailed treatment
of the relation between $r$-matrices and twist see \cite{jonas}).  It can
be proved that $r$-matrices always induces a Lie bialgebra structure
on $\lie g$. Thus, the corresponding Lie group $G$ automatically becomes a Poisson Lie group,
since the Poisson tensor obtained by integrating the Lie bialgebra structure on $\lie g$
is multiplicative. The concept of momentum map for Poisson Lie groups acting
on Poisson manifolds has been first introduced by Lu in \cite{Lu1990, Lu1991}, in the case
in which the Poisson structures of $G$, its dual $G^*$ and $M$ are fixed. In contrast to the ordinary momentum map 
it takes values in $G^*$ and the equivariance is defined in relation to the so-called dressing action of $G$
on $G^*$.
Here we introduce a slight generalization and then focus on the case in which,
in the same spirit as Drinfel'd, 
the Poisson structure on $G^*$ is induced by $r$ via the dressing action and on $M$
via the action $\varphi$. 

In
Section~\ref{sec:hopfhamiltonian} we construct a momentum map
in the setting of Hopf algebra actions and coactions and study its 
quantization. More precisely, given a
classical Hamiltonian action $\varphi : \lie g \to \Secinfty(TM)$ 
with momentum map $J: M \to G^*$
we construct a corresponding 
Hopf algebra action and we prove that $J^*$ defines a momentum map for this
action. This allows us to define the notion of Hamiltonian Hopf algebra action.
Motivated by the significance of coactions in the 
theory of quantum groups in the $C^*$-algebraic framework, we give a dual version
of the above result and prove that given $\varphi$ 
the corresponding Hopf algebra coaction $\delta_\Phi : \Cinfty(M)
  \to \Cinfty(M) \tensor \Cinfty(G)$ is also Hamiltonian.
Finally, using the UDF we obtain the quantized algebras $\Cinfty_\hbar(M)$ and we prove that
the quantum group
  coaction $\delta_\Phi : \Cinfty_\hbar(M) \to
  \Cinfty_\hbar(M) \tensor \Cinfty_\hbar(G)$ is again Hamiltonian.

\section{Preliminaries}
\label{sec:preliminaries}

Let $\lie{g}$ be a (finite-dimensional) Lie algebra and consider the
algebra $\mathcal{U}(\lie{g})[[\hbar]]$ of formal power series with
coefficients in the universal enveloping algebra
$\mathcal{U}(\lie{g})$.  It can be endowed with a (topologically free)
Hopf algebra structure, denoted by $(\mathcal{U}(\lie{g})[[\hbar]],
\Delta, \epsilon, S)$. Let us recall the definition of a Drinfel’d twist 
and its semiclassical limit, see \cite{drinfeld:1983, drinfeld:1988a}.
\begin{definition}[Twist]
\label{def:twist}
  An element $\mathcal{F}\in(\mathcal{U}(\lie{g}) \tensor
  \mathcal{U}(\lie{g}))[[\hbar]]$ is said to be a twist on
  $\mathcal{U}(\lie{g})[[\hbar]]$ if the following three conditions
  are satisfied:
  \begin{definitionlist}
  \item
      $\mathcal{F} = 1 \tensor 1 + \sum_{k=1}^\infty \hbar^k \mathcal{F}_k$.
  \item \label{cond:cocycle} $(\mathcal{F}\tensor 1)(\Delta\tensor
    1)(\mathcal{F}) = (1 \tensor
    \mathcal{F})(1\tensor\Delta)(\mathcal{F})$.
  \item
    $(\epsilon\tensor 1)\mathcal{F} = (1\tensor\epsilon)\mathcal{F} = 1$.
  \end{definitionlist}
\end{definition}
We sometimes use the notation $\mathcal{F} = \mathcal{F}^\alpha
\tensor \mathcal{F}_\alpha$.  The semiclassical limit of a twist gives
rise to a well-known structure on the Lie algebra $\lie g$ called
$r$-matrix, as proved in \cite{drinfeld:1988a} or
\cite[Thm.~1.14]{giaquinto.zhang:1998a}. In fact, we have the following
claim.
\begin{proposition}
  Given a twist $\mathcal{F}$ on $\mathcal{U}(\lie{g})[[\hbar]]$, the
  antisymmetric part of its first order is a classical $r$-matrix $r
  \in \lie{g} \wedge \lie{g}$.
\end{proposition}
Given a twist we can obtain a deformed Hopf algebra structure on
$\mathcal{U}(\lie{g})[[\hbar]]$.
\begin{proposition}
  \label{prop:defcopr}
  Let $\mathcal{F}$ be a twist on
  $\mathcal{U}(\lie{g})[[\hbar]]$. Then the algebra
  $\mathcal{U}(\lie{g})[[\hbar]]$ endowed with coproduct given by
  \begin{equation}
    \label{eq:defcopr}
    \Delta_\mathcal{F}
    :=
    \mathcal{F} \Delta \mathcal{F}^{-1}, 
  \end{equation}
  undeformed counit and antipode $S_\mathcal{F} := u_\mathcal{F} S(X)
  u_\mathcal{F}^{-1}$, where $u_\mathcal{F}:= \mathcal{F}^\alpha
  S(\mathcal{F}_\alpha)$ is again a Hopf algebra denoted by
  $\mathcal{U}_\mathcal{F} (\lie g)$.
\end{proposition}
As a consequence, the twist automatically defines a Lie bialgebra
structure.  Given a twist on the universal enveloping algebra, we can
always define a star product on any $\mathcal{U}(\lie{g})$-module
algebra. In particular, let us consider the algebra $\Cinfty(M)$ of
smooth functions on a manifold $M$ with pointwise multiplication $m_M$
and a Hopf algebra action
\begin{equation}
  \label{eq:Action}
  \Phi
  \colon
  \mathcal{U}(\lie{g}) \times \Cinfty(M)
  \longrightarrow
  \Cinfty(M)
  \colon 
  (X, f)
  \mapsto 
  \Phi(X, f) 
\end{equation}
This action can be immediately extended to formal power series, allowing the following 
result.
\begin{lemma}[Universal deformation formula]
\label{lem:startwist}
  The product defined by
  \begin{align}
    \label{eq:StarTwist}
    f \star_\mathcal{F} g
    =
    m_M(\Phi(\mathcal{F}^{-1},(f\tensor g)))
  \end{align}
  for $f, g \in \Cinfty(M)[[\hbar]]$ is an associative star product
  quantizing the Poisson structure induced by the semiclassical limit
  $r$ of $\mathcal{F}$ via the action.
\end{lemma}
We denote the deformed algebra by $\Cinfty_\mathcal{F}(M)$.
%
%
Moreover, it is important to remark that the deformed algebra
$\Cinfty_\mathcal{F}(M)$ is now a left $\mathcal{U}_\mathcal{F} (\lie
g)$-module algebra

We can give a dual version of the above discussion by using the
notions of 2-cocycles and coactions.  A more detailed discussion about
2-cocycles and their duality with twist can be found in
\cite{Bieliavsky2016}.  Consider the Hopf algebra $\Cinfty(G)$, where
$G$ is the Lie group corresponding to the finite-dimensional Lie
algebra $\lie g$.  It is known that $\Cinfty(G)$ and
$\mathcal{U}(\lie{g})$ are dually paired Hopf algebras algebras with
pairing denoted by $\langle \argument, \argument \rangle$. Thus, given
a twist $\mathcal{F}$ there corresponds an element $\gamma :
(\Cinfty(G) \tensor \Cinfty(G))[[\hbar]] \to \mathbb{K}$ on
$\Cinfty(G)[[\hbar]]$ defined by
\begin{equation}
  \gamma (f \tensor g)
  :=
  \langle \mathcal{F}^\alpha , f \rangle \langle \mathcal{F}_\alpha , g \rangle,
\end{equation}
for all $f, g \in \Cinfty(G)$. Roughly, from the condition
\refitem{cond:cocycle} mentioned in Definition~\ref{def:twist} it is
easy to see that $\gamma$ satisfies the 2-cocycle condition
\begin{equation}
  \gamma (f_{(1)} \tensor g_{(1)} ) \gamma (f_{(2)} g_{(2)} \tensor h)
  =
  \gamma (g_{(1)} \tensor h_{(1)} ) \gamma (f \tensor g_{(2)} h_{(2)}),
\end{equation}
for any $f,g,h \in \Cinfty(G)$. Here we used the Sweedler
notation. Thus, dualizing the deformed Hopf algebra
$\mathcal{U}_\mathcal{F} (\lie g)$ obtained in
Proposition~\ref{prop:defcopr} we immediately obtain a twisted Hopf
algebra denoted by $\Cinfty_\gamma (G)$ with a new associative product
$m_G^\gamma$ defined by
\begin{equation}
  \label{eq:mgamma}
  \langle \Delta_\mathcal{F} (X), f \tensor g \rangle
  =
  \langle X, m_G^\gamma (f \tensor g) \rangle
\end{equation}
As it will be used in the following we resume this structure in the following
defining.
\begin{definition}[Quantum group]
\label{def:quantumgroup}
  The quantum group corresponding to $G$ is defined to be the Hopf algebra
  $\Cinfty_\gamma (G)$ given by $(\Cinfty(G)[[\hbar]], m_G^\gamma, \Delta, \epsilon, S)$
  where the deformed product is given by \eqref{eq:mgamma} and coproduct and counit are undeformed.
\end{definition}
The deformed Hopf algebras
$\Cinfty_\gamma (G)$ and $\mathcal{U}_\mathcal{F} (\lie g)$ are again
dually paired via the same pairing.
Finally, if $\Cinfty(M)$ is a left $\mathcal{U} (\lie g)$-module
algebra via \eqref{eq:Action}, it is automatically a
right-$\Cinfty(G)$-comodule algebra (the coaction $\delta : \Cinfty(M)
\to \Cinfty(M)\tensor \Cinfty(G)$ can be easily obtained by dualizing
$\Phi$, see \cite[Prop.~1.6.11]{Majid}). In the same spirit of
Lemma~\ref{lem:startwist}, the algebra structure of $\Cinfty(M)$ can
be equivalently deformed by considering a 2-cocycle on $\Cinfty(G)$ and pushing its
deformation on $\Cinfty(M)$ via the coaction $\delta$. 



\section{Hamiltonian actions}
\label{sec:hamiltonian}

In this section we introduce the notion of Hamiltonian action in the
setting of Poisson Lie groups. This notion has been first defined in
\cite{Lu1990,Lu1991} in the case of a Poisson Lie group acting on a
Poisson manifold with both Poisson structures fixed. In our work we
are mainly interested in the case in which the Poisson structure on
the manifold is the one induced by the action. This requires a slight
generalization of the notion of Hamiltonian action.

\subsection{Dressing generators}

In the same spirit of \cite{Lu1990,Lu1991}, the notion of Hamiltonian
action relies on the definition of momentum map, which provides us of
a comparison tool between the dressing orbits and the orbit of the
considered action. For this reason, we first focus on the dressing
action and in particular on the possible descriptions of the
corresponding fundamental vector fields.

Let us consider a Lie bialgebra $\lie g$ with dual and double denoted 
by $\lie g^*$ and $\lie d$, respectively. The Lie groups $G$ and $G^*$ associated to 
$\lie g$ and $\lie g^*$, respectively, turn into Poisson Lie groups. 
Furthermore, the Lie group $D$ corresponding to the double Lie algebra $\lie d$
is called double of the Poisson Lie group $G$.

Consider $g \in G$,
$u \in G^*$ and let $u g \in D$ be their product.  Since $\mathfrak{d}
= \lie{g} \oplus\lie{g}^*$, elements in $D$ close to the unit can be
decomposed in a unique way as a product of an element in $G$ and an
element in $G^*$.  Then, there exist elements $^{u}g \in G$ and $u^g
\in G^*$ such that
\begin{equation}
  \label{eq:DressDec}
  ug 
   = 
  {}^{u}g u^g.
\end{equation}
Hence, the action of $g\in G$ on $u\in G^*$ is given by
\begin{equation}
  (u,g) \mapsto (ug)_G^*
\end{equation}
where $(ug)_G^*$ denotes the $G^*$-factor of $ug\in D$.  This defines
a left action of $G$ on $G^*$, called \emph{dressing action}.  This
action plays an important role in the context of Poisson actions since
its orbits coincides with the symplectic leaves of $G^*$ and its
linearization is the coadjoint action.
Let us denote by $\ell_X$ the corresponding fundamental vector field
for $X \in \lie g$. In the following we introduce the notion of dressing generators,
which are one-forms that give us the fundamental vector fields $\ell_X$ if contracted 
with the Poisson bitensor. As it will be seen in the next sections these forms 
are in general not globally defined, so we use the notation $\lforms{G^*}$ to 
denote local forms on $G^*$.
\begin{definition}[Dressing generator]
  \label{def:DressingGen}
  The map $\alpha: \lie{g} \to \lforms{G^*} : X \mapsto \alpha_X$ is
  said to be dressing generator with respect to the Poisson structure
  $\pi$ on $G^*$ if the fundamental vector field $\ell_X$ of the
  dressing action can be written as
  \begin{equation}
    \label{eq:DressShift}
    \ell_X
    =
    \pi^\sharp (\alpha_X)
  \end{equation}
  and satisfies
  \begin{align}
    \label{eq:AlgMorph}
    \alpha_{[X, Y]} 
    &= 
    [\alpha_X , \alpha_Y]_{\pi_\ell},
    \\
    \label{eq:MC}
    \D \alpha_X 
    &= 
    \alpha \wedge \alpha \circ \delta(X).
  \end{align}
  Here $\delta$ denotes the Lie bialgebra structure on $\lie g$.
\end{definition}

\begin{remark}
  The first example of dressing generators with respect to the standard dual
  Poisson structure $\pi_*$ is given by the left-invariant one-forms
  corresponding to the element $X$, as proved in \cite[Appendix 2,
   page 66]{Kosmann-Schwarzbach2004}. 
  As already mentioned, the dressing generators with respect to a generic Poisson 
  structure on $G^*$ are in general not globally defined (a concrete example is computed
  in the next section). However, the contraction with the Poisson tensor still gives rise
  to a smooth vector field.
\end{remark}
Here we are interested to the case in which $\lie g$ is endowed with
an $r$-matrix and we consider the Poisson structure $\pi_\ell$ induced
by the infinitesimal dressing action $\ell\colon \lie g \to \Secinfty
(T G^*)$ via
\begin{equation}
  \label{eq:pilambda}
  \pi_\ell
  =
  r_{ij}\ell_{X_i} \wedge \ell_{X_j} .
\end{equation}
This is a natural candidate since the contraction of $\pi_\ell$ with
one-forms satisfying \eqref{eq:AlgMorph}-\eqref{eq:MC} gives
 rise automatically to an infinitesimal Poisson action, as proved in
the following Lemma.
\begin{lemma}
  \label{lem:poissonaction}
  Given a map $\alpha\colon \lie{g} \to \lforms{G^*}$ satisfying
  \eqref{eq:AlgMorph}-\eqref{eq:MC} then we have:
  \begin{lemmalist}
  \item The map $\lie g \ni X \mapsto \pi_\ell^\sharp ( \alpha_X) \in
    \Secinfty(TG^*)$ is a Lie algebra morphism
  \item The map $\lie g \ni X \mapsto \pi_\ell^\sharp ( \alpha_X) \in
    \Secinfty(TG^*)$ is an (infinitesimal) Poisson action.
  \end{lemmalist}
\end{lemma}
\begin{proof}
  Let us compute:
    \begin{align*}
      \pi_\ell^\sharp (\alpha_{[X, Y]})
      &\ot{\eqref{eq:AlgMorph}}{=}
      \pi_\ell^\sharp ([\alpha_X , \alpha_Y]_{\pi_\ell})\\
      &\ot{(*)}{=}
      [\pi_\ell^\sharp(\alpha_X), \pi_\ell^\sharp(\alpha_Y)].
    \end{align*}
    In $(*)$ we used the fact that $\pi_\ell^\sharp$ is a Lie algebra
    morphism with respect to the Lie bracket of one-forms $[a, b]_{\pi_\ell} =
    \Lie_{\pi_\ell^\sharp(a)}b - \Lie_{\pi_\ell^\sharp(b)}a - \D
    \pi_\ell (a,b)$.
  Furthermore, we have:
    \begin{align*}
      \wedge^2 \pi_\ell^\sharp (\alpha\wedge \alpha \circ \delta (X))
      &\ot{\eqref{eq:MC}}{=}
      \wedge^2 \pi_\ell^\sharp (\D \alpha_X)
      \\
      &\ot{(*)}{=}
      \D_{\pi_\ell} \pi_\ell^\sharp (\alpha_X).
    \end{align*}
    In $(*)$ we used $\D_\pi(\wedge^p \pi^\sharp)(\xi)) =
    (\wedge^{p+1}\pi^\sharp)(\D \xi)$.  
\end{proof}

\begin{example}[Dressing generators on $ax+b$]
  Let  us denote by $\lie{s}$ the Lie algebra with basis $H$, $E$ and 
  commutation relation
  \begin{equation}
    [H , E] 
    = 
    2E,
  \end{equation}
  also known as $ax + b$. The corresponding group is denoted by $S$
  and we consider the dressing action $S \times S^* \to S$. Then we
  have that The dressing generators with respect to $\pi_\ell$ are given by the local forms
  \begin{equation}
      \label{eq:dressinglambda}
      \alpha_{H} 
      =
      \frac{1}{y}\D x
      \quad\text{and}\quad
      \alpha_{E} 
      = 
      \frac{1}{2y}\D y.
  \end{equation}
  The complete discussion of this example can be found in the Appendix~\ref{ex:Ax+B}.
\end{example}

%
\subsection{Hamiltonian actions}

Using the notion of dressing generator we give a new definition of
Hamiltonian action in this context.
%

\begin{definition}[Momentum map]
  Let $\Phi : G \times M \to M$ be an action of $(G, \pi_G)$ on $(M,
  \pi)$ and $\alpha_X$ the dressing generator with respect to a Poisson structure
  $\pi_{G^*}$ on $G^*$.
  \begin{definitionlist}
  \item  A momentum map for $\Phi$ is a map
  $J : M \to G^*$ such that
  \begin{equation}
    \varphi(X) 
    = 
    \pi^\sharp (J^* (\alpha_X)),
  \end{equation}
  where $\varphi(X)$ is the fundamental vector field of $\Phi$.  
  In other words, $J$ is defined by the commutativity of the
  following diagram:
  \begin{equation}
    \label{diag:momentumdiag}
    \begin{gathered}
    \includegraphics[height=14ex]{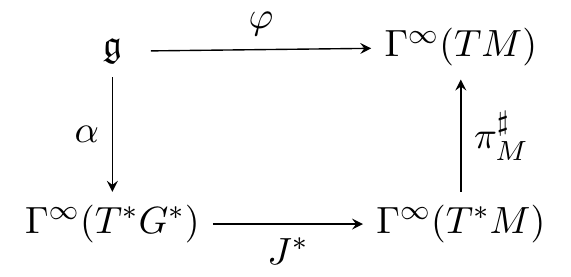}
    \end{gathered}
  \end{equation}
  \item A map $J : M \to G^*$ is said to be $\ell$-equivariant if it
    intertwines the fundamental vector field $\varphi(X)$ and the
    dressing action $\ell_X $ for any $X$.
  \end{definitionlist}
\end{definition}
\begin{lemma}
  \label{lem:PoissonEq}
  The momentum map $J$ defined above is $\ell$-equivariant if and only
  if is Poisson. 
\end{lemma}
\begin{proof}
  Let us consider generic Poisson structures $\pi$ on $M$ and
  $\pi_{G^*}$ on $G^*$. Thus, $J$ is a Poisson map if and only if
  \begin{equation*}
    J_*(\pi^\sharp (J^*(\alpha))) 
    = 
    \pi_{G^*}^\sharp(\alpha).
  \end{equation*}
  Let $\alpha$ be the dressing generator corresponding to $\pi_{G^*}$.
  Thus $\pi_{G^*}^\sharp(\alpha_X) = \ell_X $ and $\pi^\sharp (J^*
  (\alpha_X)) = \varphi(X)$ and the equation above coincides with the
  $\ell$-equivariance.
\end{proof}
Now the notion of Hamiltonian follows naturally:
\begin{definition}[Hamiltonian action]
  \label{def:Hamiltonian}
  An action $\Phi$ of $(G, \pi_G)$ on $(M,\pi_M)$ is said to be
  Hamiltonian if it is Poisson and is generated by a $\ell$-equivariant momentum map
  $J : M \to G^*$.
\end{definition}
Since in the following we mainly use the infinitesimal action
$\varphi$, we say that it is Hamiltonian whenever the corresponding
$\Phi$ is Hamiltonian.
\begin{remark}
  \begin{remarklist}
   \item If we choose the standard dual Poisson structure
  on $G^*$, the dressing generators are the left-invariant one-forms
  and the above definition boils down to the definition of momentum
  map and Hamiltonian action given by Lu in \cite{Lu1990, Lu1991}.
  \item Let $\lie g$ be a triangular Lie algebra with $r$-matrix $r$, acting
  on a manifold $M$ by $\varphi: \lie g \to \Gamma^\infty(TM)$. We
  denote by $\pi_r$ the Poisson structure induced by $r$ via
  \begin{equation}
   \pi_r
   =
   r^{ij} \varphi(X_i) \wedge \varphi(X_j).
  \end{equation}
  In this case the action $\varphi$ and its global corresponding are automatically
  Poisson. (The proof is the same as the one given in Lemma~\ref{lem:poissonaction}).
  \end{remarklist}
%
 \end{remark}

\begin{example}[Dressing action]
  The easiest example is given by the dressing action. Here the
  momentum map is just the identity.
\end{example} 
\begin{example}[Coadjoint action]
  Let us consider the Poisson structure $\pi_r$ induced by the
  coadjoint action.  Notice that $\pi_r$ does not coincide with the
  linear one.
  As proved in \cite[Section 3.3]{Alekseev2013} one can define a map
  $j: \mathfrak{g}^*\to \lie{d}$ by $j(\xi) = \xi - r(\xi, \argument)$. Thus, the modified exponential is
  given by
    $$
        \mathrm{Exp} : \mathfrak{g}^* \to G^* 
        : \mathrm{Exp}(\xi) := \pr_{G^*}(\exp (j(\xi))).
    $$
  In contrast to the usual exponential map
  it intertwines the coadjoint action with the dressing action, hence it takes
  symplectic leaves to symplectic leaves. In other words, we have
    $$
        \ell_X = \mathrm{Exp}_*\varphi(X).
    $$
        If $G$ is compact with the Lu-Weinstein Poisson structure \cite{Lu1990b}, 
    $\mathrm{Exp}$ is a global 
    diffeomorphism (see \cite[Remarks 3.5]{Alekseev2013}).
    An easy computation shows that $\mathrm{Exp}$ is a momentum map for the coadjoint action.
\end{example}
\begin{remark}
  From the above example we can construct other Hamiltonian actions.
  Given a standard momentum map $\mu: M \to \lie g^*$ which is
  $\ad^*$-equivariant we can always construct a momentum map $J : M
  \to G^*$ by composing $\mu$ and $\mathrm{Exp}$.  For instance,
  observing that $r^\sharp \colon \lie g^* \to \lie g$ intertwines
  adjoint and coadjoint actions
  we can conclude that
  the adjoint action is Hamiltonian with momentum map given by the
  composition of $r^\sharp$ with $\mathrm{Exp}$.
\end{remark}


\begin{remark}
  The reduction can been obtained with various techniques (see
  e.g. \cite{chiara1}). We here remark that the preimage $C =
  J^{-1}(\{ 0 \})$ of a $\ell$-invariant momentum map is a coisotropic
  submanifold and $\mathcal{I}_C$ the corresponding vanishing
  ideal. Thus the reduced algebra can be easily obtained by the
  quotient $\mathcal{B}_C/ \mathcal{I}_C$ where $\mathcal{B}_C = \{ f \in \Cinfty (M) \vert \{f,
  \mathcal{I}_C\} \subseteq \mathcal{I}_C \}$.
 \end{remark}

\section{Hamiltonian Hopf algebra (co)actions}
\label{sec:hopfhamiltonian}

In this section we aim to give a definition of Hamiltonian (co)action
in the setting of Hopf algebra (co)actions and a possible quantization procedure. 
In the same spirit of
Definition~\ref{def:Hamiltonian}, given an Hopf algebra action $\Phi$,
a momentum map has to be an intertwiner between dressing action and
$\Phi$. In order to introduce this notion we first prove that given a classical Hamiltonian
action we can always associate a Hopf algebra action and construct, out of the
classical momentum map, the desired intertwiner.
%

First, we observe that any Lie algebra action gives rise to a Hopf
algebra action.
\begin{lemma}
\label{lem:dressinghopf}
  Consider the infinitesimal action $\varphi : \lie g \to
  \Secinfty (T M)$.  This is equivalent to a Hopf algebra action
  $\Phi: \mathcal{U}(\lie g) \times \Cinfty (M) \to \Cinfty
  (M)$ by setting
  \begin{equation}
    \label{eq:Phi}
    \Phi(X, f) 
    :=
    \Lie_{\varphi_X}f,
  \end{equation}
  where $\Lie$ denotes the Lie derivative.
  Equivalently, it defines a Hopf algebra coaction $\delta_\Phi:
  \Cinfty (M) \to \Cinfty (M) \tensor \Cinfty (G)$
\end{lemma}
\begin{proof}
  The Lie algebra elements act as derivations of
  $\Cinfty(M)$, thus $\Phi$ defines a Lie algebra action
  $\varphi\colon \lie{g}\to \Secinfty(TM)$. Since the elements of
  $\lie{g}$ generate $\mathcal{U}(\lie{g})$, the action $\Phi$ is
  given by differential operators with order determined by the natural
  filtration of the universal enveloping algebra.  Conversely, every
  Lie algebra action $\varphi$ of $\lie{g}$ on $M$ determines via the
  fundamental vector fields $\varphi_X \in \Secinfty(TM)$ a
  representation of $\lie{g}$ on $\Cinfty(M)$ by derivations which
  therefore extends to a Hopf algebra action $\Phi$ as above. The
  action $\Phi$ and the coaction $\delta_\Phi$ are always
  equivalent.
\end{proof}
In particular, given the infinitesimal dressing action
$\ell \colon \lie{g}\to \Secinfty(T G^*)$ we obtain the Hopf 
algebra action $\Lambda\colon \mathcal{U}(\lie g) \times \Cinfty (G^*) \to \Cinfty
  (G^*)$ by setting:
\begin{equation}
    \label{eq:Lambda}
    \Lambda(X, f) 
    :=
    \Lie_{\ell_X}f.
  \end{equation}
We denote by $\delta_\Lambda$ the corresponding Hopf algebra coaction.
As a next step we lift the notion of dressing generator to the setting
of Hopf algebra actions. We observe that, given the Lie algebra
representation $\alpha: \lie g \to \lforms{G^*}$, we can define
another Hopf algebra action by using the Lie derivative in the
direction of a one-form $\mathcal{L}_{\alpha}$ which has been defined
by Bhaskara and Viswanath \cite{Bhaskara1988}.  In particular, for
$f\in \Cinfty (G^*)$
\begin{equation}
  \mathcal{L}_{\alpha} f
  =
  \Lie_{\pi^\sharp (\alpha)} f.
\end{equation}
More precisely, we have:
\begin{lemma}
  \label{lem:LambdaLie}
  Given a dressing generator $\alpha: \lie g \to \lforms{G^*}$,
  the corresponding map given by $\mathcal{U}(\lie g) \times \Cinfty
  (G^*) \to \Cinfty (G^*): (X, f) \mapsto \mathcal{L}_{\alpha_X} f$ is
  a Hopf algebra action. Furthemore we have
  \begin{equation}
    \label{eq:LambdaLie}
    \Lambda(X, f) 
    =
    \mathcal{L}_{\alpha_X} f,
  \end{equation}
  where $\Lambda (X, f)$ is given by \eqref{eq:Lambda}.
 \end{lemma}
\begin{proof}
  First, as in Lemma~\ref{lem:dressinghopf} the map $\mathcal{U}(\lie g) \times \Cinfty (G^*) \to \Cinfty
  (G^*): (X, f) \mapsto \mathcal{L}_{\alpha_X}$ immediately satisfies
  the condition to be a Hopf algebra action.
  Also, from the definition of dressing generator
  we have $\ell_X = \pi_\ell^\sharp (\alpha_X)$. Thus
  \begin{align*}
    \Lambda(X, f) 
    &= 
    \Lie_{\ell_X}f
    \\
    &=
    \Lie_{\pi^\sharp_\ell (\alpha_X)}f
    \\
    &=
    \mathcal{L}_{\alpha_X}f.
  \end{align*}
\end{proof}
%
%
%
%
%
%
Now, let us consider a Hamiltonian action $\varphi : \lie g \to
\Secinfty(T M)$ with momentum map $J: M \to G^*$. Notice that its
pullback of functions $J^*: \Cinfty(G^*) \to \Cinfty(M)$ is an algebra
morphism.  With an abuse of notation, we also refer to $J^*$ as the
pullback of forms. Since the latter is always defined, we can extend
$J$ to a map $J^*$ acting on $\mathcal{L}_\alpha$ by
\begin{equation}
  \label{eq:JJ}
  J^* \mathcal{L}_\alpha 
  :=
  \mathcal{L}_{J^*\alpha} \circ J^*.
\end{equation}

\begin{theorem}
  Let $\varphi : \lie g \to \Secinfty(TM)$ be an Hamiltonian action
  with momentum map $J: M \to G^*$ and consider the corresponding Hopf
  algebra action $\Phi : \mathcal{U}(\lie g) \times \Cinfty (M) \to
  \Cinfty (M)$ given by $\Phi(X) = \Lie_{\varphi(X)}$.  Then we have:
  \begin{theoremlist}
   \item The pullback $J^*: \Cinfty(G^*) \to \Cinfty(M)$ of $J$ intertwines $\Phi$ and 
   the Hopf algebra action $\Lambda$ corresponding to the dressing action via \eqref{eq:Lambda}.
   \item The pullback $J^*: \Cinfty(G^*) \to \Cinfty(M)$ of $J$ intertwines the corresponding Hopf
  algebra coaction $\delta_\Phi$ and the Hopf algebra coaction $\delta_\Lambda$ corresponding to the dressing action.
  \end{theoremlist}
\end{theorem}
\begin{proof}
  The two claims above can be rephrased by saying that $J^*$ defines a $\mathcal{U}(\lie g)$-module
  algebra morphism and $\Cinfty(G)$-comodule algebra morphism. 
  \begin{theoremlist}
   \item We already observed that $J^*: \Cinfty(G^*) \to \Cinfty(M)$ is an
  algebra morphism.  Thus, we only need to prove that it is a module
  morphism, i.e. the commutativity of the following diagram:
  \begin{equation}
    \label{diag:module}
    \begin{gathered}
      \includegraphics[height=16ex]{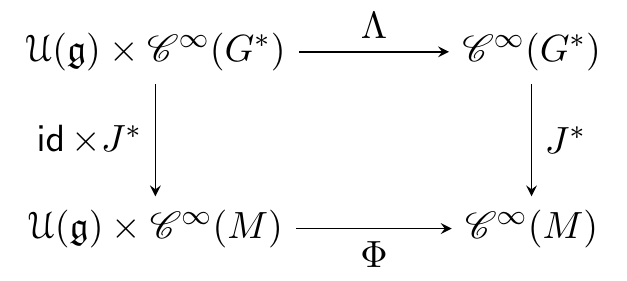}
    \end{gathered}
  \end{equation}
  In other words, we need to prove 
  \begin{equation}
    \label{eq:equiv}
    \Phi(X, J^* f)
    =
    J^*(\Lambda(X, f)).
  \end{equation}
  Using \eqref{eq:JJ} we can easily compute:
  \begin{align*}
    J^*(\Lambda(X, f))
    &=
    J^* (\mathcal{L}_{\alpha_X} f)
    \\
    &=
    \mathcal{L}_{J^*\alpha_X}  J^* f
    \\
    &=
    \Lie_{\pi^\sharp (J^* (\alpha_X))}J^* f
    \\
    &=
    \Lie_{\varphi(X)}J^* f
    \\
    &=
    \Phi(X, J^* f).
  \end{align*}
  Here we used the fact that, from Definition~\ref{def:Hamiltonian},
  we have $\varphi(X) = \pi^\sharp (J^* (\alpha_X))$.
   \item Given the Hopf algebra action $\Phi$ we can always find the
   corresponding Hopf algebra coaction $\delta_\Phi$, as discussed in
   Section~\ref{sec:preliminaries}.  Thus we can immediately state the
   dual version of the above claim. In fact, dualizing the commutative diagram 
   \eqref{diag:module} we immediately get the following commutative diagram
  \begin{equation}
    \label{diag:comodule}
    \begin{gathered}
      \includegraphics[height=16ex]{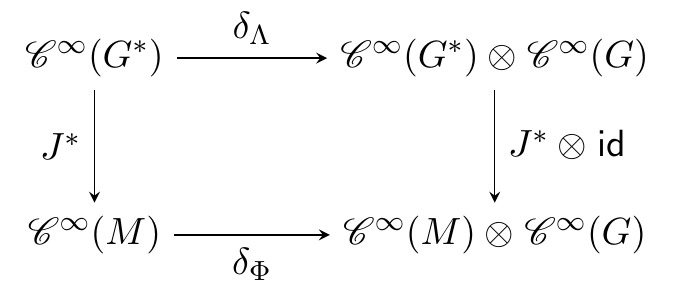}
    \end{gathered}
  \end{equation}
  which gives the comodule morphism condition $\delta_\Phi \circ J^* =
  (J^* \tensor \id) \circ \delta_\Lambda$.  Since $J^*$ is an algebra
  morphism the claim is proved.
  \end{theoremlist}
\end{proof}
Finally, the above discussion motivates the following definition.
Let $\Cinfty (M)$ be a $\mathcal{U}(\lie g)$-module algebra
where the module structure is given by a generic Hopf algebra action 
$\Phi : \mathcal{U}(\lie g) \times \Cinfty (M) \to \Cinfty (M)$. Equivalently,
$\Cinfty (M)$ is endowed with a $\Cinfty(G)$-comodule algebra structure.
Furthermore, given the dressing action $\ell$ we showed that $\Cinfty(G^*)$
automatically turns into a $\mathcal{U}(\lie g)$-module algebra where the 
Hopf algebra action $\Lambda$ is given by \eqref{eq:Lambda} (and equivalently
into a $\Cinfty(G)$-comodule algebra).
\begin{definition}[Hamiltonian (co)action]
\label{def:hopfhamilt}
  \begin{definitionlist}
  \item A Hopf algebra action $\Phi : \mathcal{U}(\lie g) \times
    \Cinfty (M) \to \Cinfty (M)$ is said to be Hamiltonian if there exist a
    $\mathcal{U}(\lie g)$-module algebra morphism, called momentum map, $\mathbf{J}: \Cinfty(G^*) \to
    \Cinfty(M)$. In other words, $\Phi$ is Hamiltonian if it allows a map $\mathbf{J}$
    satisfying the following condition:
     \begin{equation}
     \Phi(X, \mathbf{J} f)
     =
     \mathbf{J}(\Lambda(X, f)).
  \end{equation}
  \item A Hopf algebra coaction $\delta_\Phi : \Cinfty(M) \to
    \Cinfty(M) \tensor \Cinfty(G)$ is said to be Hamiltonian if there exist
    $\Cinfty(G)$-module algebra morphism $\mathbf{J}$, called momentum
    map, which intertwines it with the Hopf algebra coaction
    $\delta_\Lambda$ corresponding to the dressing action.
  \end{definitionlist}
\end{definition}

\subsection{Quantum Hamiltonian coactions via 2-cocycles}


In this section we prove that, using Drinfeld approach, we obtain a
quantization of the Hamiltonian coactions as in
Definition~\ref{def:hopfhamilt}. Since actions and coactions are
completely equivalent we here prefer to focus only on the coaction
case.

Let us consider a twist $\mathcal{F}$ on $\mathcal{U}(\lie g)$ with
corresponding 2-cocycle $\gamma$ on $\Cinfty(G)$.  As seen in
Definition~\ref{def:quantumgroup}, the 2-cocycle $\gamma$ induces a
deformed product $\star_\gamma$ and we denote by $\Cinfty_\hbar(G)$
the corresponding quantum group.
Furthermore, we obtain a deformed product on the comodule algebras
$\Cinfty(M)$ and $\Cinfty(G^*)$.  More precisely, the action $\varphi
: \lie g \to \Secinfty(TM)$ induces a star product $\star_\varphi$ on
$M$ whose semiclassical limit is the Poisson structure $\pi_r$ induced
by $r$ via $\varphi$.
Similarly, the dressing action $\ell: \lie g \to \Secinfty(TG^*)$
induces a star product $\star_\ell$ on $G^*$.  Let us denote by
$\Cinfty_\hbar(G^*)$ the deformed algebra given by the pair
$(\Cinfty(G^*)[[\hbar]], \star_\ell)$ and by $\Cinfty_\hbar(M)$ the
pair $(\Cinfty(M)[[\hbar]], \star_\varphi)$. Notice that
$\Cinfty_\hbar(G^*)$ and $\Cinfty_\hbar(M)$ are now
$\Cinfty_\hbar(G)$-comodule algebras. In other words, the coactions
\begin{equation}
  \label{eq:coactions}
  \delta_\Phi : \Cinfty_\hbar(M) \to \Cinfty_\hbar(M) \tensor \Cinfty_\hbar(G)
  \quad\text{and}\quad
  \delta_\Lambda : \Cinfty_\hbar(G^*) \to \Cinfty_\hbar(G^*) \tensor \Cinfty_\hbar(G)
\end{equation}
are morphisms of algebras.  Thus we can state our main result.
\begin{theorem}
  Let $\varphi : \lie g \to \Secinfty(TM)$ be an Hamiltonian action
  with momentum map $J: M \to G^*$.  Then the corresponding quantum group
  coaction $\delta_\Phi : \Cinfty_\hbar(M) \to
  \Cinfty_\hbar(M) \tensor \Cinfty_\hbar(G)$ is Hamiltonian in the
  sense of Definition~\ref{def:hopfhamilt}.
\end{theorem}
\begin{proof}
  Since in the Drinfeld approach the coactions do not change but they
  only intertwine different algebraic structures, the classical
  momentum map is still a comodule morphism as in
  Lemma~\ref{lem:Jcomodule}.  More explicitly, the diagram
  \begin{equation}
    \begin{gathered}
      \includegraphics[height=16ex]{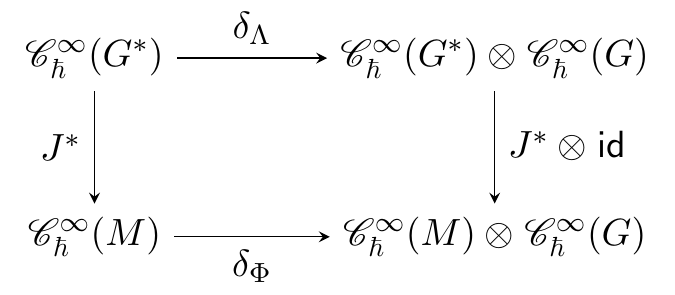}
    \end{gathered}
  \end{equation}
  commutes. Thus, we only need to prove that $J^* : \Cinfty_\hbar(G^*)
  \to \Cinfty_\hbar(M)$ is a morphism of algebras. This can be
  immediately checked by using the UDF \eqref{eq:StarTwist} and
  Lemma~\ref{lem:Jmodule}.  We can extend the action \eqref{eq:Lambda}
  by
  \begin{equation}
    \Lambda(\mathcal{F}, f\tensor g) 
    =
    \Lambda(\mathcal{F}_\alpha, f) \tensor \Lambda(\mathcal{F}^\alpha, g).
  \end{equation}
  As a consequence, we have:
  \begin{align*}
    J^* (f \star_\ell g) 
    &=
    J^* (m (\Lambda(\mathcal{F}, f\tensor g)))
    \\
    &=
    J^* (m (\Lambda(\mathcal{F}_\alpha, f) , \Lambda(\mathcal{F}^\alpha, g)))
    \\
    &=
    (J^* \Lambda(\mathcal{F}_\alpha, f) )(J^*  \Lambda(\mathcal{F}^\alpha, g))
    \\
    &\ot{\eqref{eq:equiv}}{=} 
    \Phi(\mathcal{F}_\alpha, J^*f)\Phi(\mathcal{F}^\alpha, J^*g)
    \\
    &=
    m(\Phi(\mathcal{F}^{-1}, J^*f\tensor J^*g))
    \\
    &=
    J^* f \star_\varphi J^* g .
  \end{align*}
\end{proof}

\appendix

\section{Dressing generators on $ax+b$}
  \label{ex:Ax+B}

In this appendix we discuss a concrete example of dressing generators.
Let $\lie{s}$ be the Lie algebra with basis $H$, $E$ and commutation
relation
\begin{equation}
  [H , E] 
  = 
  2E ,
\end{equation}
also known as the Lie algebra $ax+b$. Consider the triangular $r$-matrix $r = H \wedge E$. This induces
the Lie bialgebra structure on $\lie g^*$:
\begin{align*}
  \delta (H) 
  &= [r, H \otimes 1 + 1 \otimes H] \\
  &= H \otimes [E, H] - [E, H] \otimes H \\
  &= -2 H \wedge E    ,
\\
\delta (E) 
  &= [r, E \otimes 1 + 1 \otimes E] \\
  &= 0.    
\end{align*}
As a consequence, the dual basis $H^*$, $E^*$
satisfies the following commutation relation:
\begin{equation}
  \label{eq:Dual}
        [H^*, E^*] 
	= 
	-2 H^* . 
\end{equation}
Note that the element $r$ corresponds to the Poisson structure
associated to the bilinear symplectic structure $\omega$ on $\lie{s}$
defined by $\omega(H,E) := 1$. Within this set up the Lie algebra
structure (\ref{eq:Dual}) on $\lie{s}^*$ is simply obtained by
transporting the Lie bracket on $\lie{s}$ to $\lie{s}^*$ under the
linear musical isomorphism $^{\flat} : \lie{s}\to \lie{s}^* : X
\mapsto \bem{X} := \iota_X \omega$ i.e.
\begin{equation}
  [\bem{X} , \bem{Y}]_{\lie{s}^*}
  := 
  \bem{[X,Y]}.
\end{equation}  
In our case we have:
\begin{equation}
  \label{eq:Flat}
  \bem{H} 
  = 
  E^* 
  \qquad  
  \bem{E} 
  =
  -H^*
\end{equation}
The double $\lie{g} := D(\lie{s})$ is given by the vector space
$\lie{s} \oplus \lie{s}^*$ equipped with the following Lie brackets (using the
notation induced by musical isomorphism)
\begin{equation}
  \begin{split}
    \label{eq:Double}
	  [H,E] &= 2E,  \quad
	  [\bem{H} , \bem{E}] = 2 \bem{E},  \quad
	  [\bem{H} , H] = 2(\bem{H} - H), \\ 
	  [H , \bem{E}] &= 2E, \quad
	  [E , \bem{E}] = 0, \quad
	  [E , \bem{H}] = -2 \bem{E}.
  \end{split} 
\end{equation}
We observe that the the first derivative $\lie g' :=[\lie g,\lie g]$
is spanned by $E$, $\bem{E}$ and $F:= \bem{H}-H$ and admits the
table:
\begin{equation*}
  [E,F] = 2E-2\,\bem{E} =:Z, \quad [E,Z]=[F,Z]=0.
\end{equation*}
Thus, $\lie g'$ is isomorphic to the Heisenberg algebra $\lie h_1:=
V\oplus \mathbb{R}Z$ associated to the symplectic plane $(V, \Omega)$
spanned by $E$ and $F$ and structured by
\begin{equation*}
  [v+zZ,v'+z'Z]
  =
  \Omega(v,v')Z
  \quad\text{with}\quad 
  v,v'\in V\quad\text{and}\quad
  \Omega(E,F):=1.
\end{equation*}
In this setting, the double $D(\lie s)$ can be viewed as the
semidirect product of the Lie algebra $\lie h_1$ with the abelian Lie
algebra $\mathbb{R} H$:
\begin{equation*}
  D(\lie s)
  \simeq
  \mathbb R\ltimes_\rho\lie h_1
\end{equation*}
whose Lie algebra homomorphism
\begin{equation*}
  \rho: \mathbb{R}H \to \Der(\lie h_1)
\end{equation*}
is  defined in the basis ${E,F,Z}$ by
\begin{equation*}
  \rho(H)
  :=
  \left(\begin{matrix} 
    2 & 0 & 0 \\ 0 & -2 & 0 \\ 0 & 0 & 0 
  \end{matrix}\right).
\end{equation*}
\begin{lemma}
  \label{lem:DoubleGroup}
  Let $\lie s = ax + b$. Then we have:
  \begin{lemmalist}
  \item The connected simply connected Lie group $ G := D(\lie{s}), $
    with Lie algebra given by the vector space $ \lie{g} := D(\lie{s})
    := \lie s \oplus \lie{s}^* $ with Lie algebra structure given by
    \eqref{eq:Double}, is diffeomorphic to the product manifold:
    \begin{equation}
      G = \mathbb{R} \times V \times \mathbb{R}.
    \end{equation}
  \item Within this model, the group law is given by
    \begin{equation}
      \label{eq:DoubleW}
      (a,v,z)\cdot (a',v',z')
      = 
      (a + a', v + e^{2a B}v', z + z' + \frac{1}{2} \Omega(v, e^{2a B}v'))
    \end{equation}
    where 
    \begin{equation}
      B
      :=
      \frac{1}{2}\rho(H)\at{V}
      =
      \left(
      \begin{array}{cc}
      1&0
      \\
      0&-1
      \end{array}
      \right)
      \quad
      \mbox{in basis}\quad\{E,F\}  .
    \end{equation}
    \item Realizing the Lie algebra $\lie{g}$ as 
      \begin{equation}
	\lie{g} 
	= \mathbb{R} H \oplus V \oplus \mathbb{R} Z
	= \{(a_0,v_0,z_0)\},
      \end{equation}
      the exponential mapping is given by
      \begin{equation}
        \label{eq:ExpD}
        \exp(a_0,v_0,z_0)
        =
        \left (a_0 , \frac{1}{2a_0}(e^{2a_0 B}-I) B v_0 , z_0 + \frac{1}{4a_0} \Omega( B v_0 , v_0)
        + \frac{1}{8a_0^2} \Omega(v_0 , e^{2a_0 B}v_0)\right).
      \end{equation}
  \end{lemmalist}
\end{lemma}
\begin{proof}
  The connected simply connected Lie group $H_1$ corresponding to
  $\lie{h}_1$ can be modelled on $V\times \mathbb R Z$ with group law
  given by
  \begin{equation}
    (v,z)\cdot(v',z')
    = (v + v', z + z'+ \frac{1}{2}\Omega(v,v')).
  \end{equation}
  Within this setting, we observe that the symplectic group
  $\group{Sp}(V,\Omega)$ (which in our two-dimensional case just
  coincides with the group $\group{SL}_2(\mathbb R)$) acts by
  centre-fixing group-automorphisms on $H_1$ under:
  \begin{equation}
    R : \group{Sp}(V, \Omega) \times H_1 \to H_1 
    : (\mathbf{a} , (v,z))\mapsto R_\mathbf{a}(v,z)
    := (\mathbf{a}(v), z).
  \end{equation}
  Every sub-group $A$ of $\group{Sp}(V,\Omega)$ therefore determines
  the semi-direct product group 
  \begin{equation}
    G := A\ltimes_RH_1
  \end{equation}
  modelled on the Cartesian product $G = A\times H_1$ with group law
  defined by ($\mathbf{a},\mathbf{a}'\in A$):
  \begin{equation}
    (\mathbf{a},v,z)\cdot(\mathbf{a}',v',z')
    :=
    (\mathbf{a} \cdot \mathbf{a}', (v,z) \cdot R_{\mathbf{a}}(v',z'))
    =(\mathbf{a} \cdot \mathbf{a}', v + \mathbf{a}(v') , z + z'+ \frac{1}{2}\Omega(v,\mathbf{a}(v'))).
  \end{equation}
  In the case
  \begin{equation}
    A 
    := \left\lbrace \exp(2a B)
    = \left(\begin{matrix} e^{2a} & 0 \\   0 & e^{-2a} \end{matrix}\right) \right\rbrace_{a\in\mathbb R}
  \end{equation}
  the semi-direct product is therefore the Lie group 
  \begin{equation}
    G = \mathbb R \times H_1
  \end{equation}
  with group law given by (\ref{eq:DoubleW}).  One then readily
  verifies that the given expression in (\ref{eq:ExpD}) satisfies the
  condition $\exp t(a_0,v_0,z_0)\cdot \exp s(a_0,v_0,z_0) =
  \exp(t+s)(a_0,v_0,z_0)$ for all $s,t\in\mathbb R$. The fact that
  $B^2 = \mathbb I$ then implies $\frac{d}{dt}\at{t=0} \exp
  t(a_0,v_0,z_0)=(a_0,v_0,z_0)$.  The computation of the Lie algebra
  of $G$ is then performed using the expression of the above
  exponential mapping (\ref{eq:ExpD}). It identifies with the one of
  $\lie{g}$.
\end{proof}
We now pass to realize $\lie{s}$ and $\lie{s}^*$ in the double $G$.
For this we start from expressing the generators at the Lie algebra
level:
\begin{equation}
    H^* = -(\frac{1}{2}Z+E)\quad\text{ and}\quad E^* = H + F .
\end{equation}
The coordinates on $\lie{s}^*$ are given by
$
    (\nu,\kappa)_*
    := 
    \exp \nu   E^* \exp \kappa  H^*
$
where 
\begin{equation}
  \exp \kappa H^* = \exp \kappa(-E-\frac{1}{2}Z)=(0, -\kappa E, -\frac{\kappa}{2})
  \quad 
  \text{and}
  \quad
  \exp \nu  E^* =\left(\nu, \frac{1}{2}(e^{-2\nu}-1)F, 0 \right)  .
\end{equation}
Using the group law (\ref{eq:DoubleW}) we get
\begin{equation}
  \label{eq:DualCoord}
  (\nu, \kappa)_{*}
  =
  \left(\nu, -\kappa e^{2\nu}E,\frac{1}{2}(e^{-2\nu}-1)F, -\frac{\kappa}{4}(1+e^{2\nu})\right)  .
\end{equation}
Similarly, we have
\begin{equation}
  \label{eq:ANCoord}
  (a,n) 
  := 
  \exp(aH)\exp(nE)
  =
  \left(a, e^{2a}nE ,0\right)  .
\end{equation}
\begin{lemma}
  Let us consider the dressing action $S \times S^* \to S$. Then we
  have
  \begin{lemmalist}
  \item The dressing generators with respect to the standard dual Poisson structure
  $\pi_*$ are given by the left-invariant forms
    \begin{equation}
      \label{eq:dressingstar}
      \alpha_{H}= -\frac{1}{y+1}\D x 
      \quad\text{and}\quad
      \alpha_{E}= \frac{1}{2(y+1)}\D y
    \end{equation}
    \item The dressing generators with respect to $\pi_\ell$ are given by the local forms
    \begin{equation}
      \label{eq:dressinglambda}
      \alpha_{H} 
      =
      \frac{1}{y}\D x
      \quad\text{and}\quad
      \alpha_{E} 
      = 
      \frac{1}{2y}\D y
    \end{equation}
    \end{lemmalist}
\end{lemma}
\begin{proof}
  The first step consists in computing the fundamental vector field of
  the dressing action by using the realization obtained above of $\lie s$ and $\lie s^*$
  in terms of the double. More explicitely, using the coordinates
  \eqref{eq:DualCoord}-\eqref{eq:ANCoord} and the group law
  \eqref{eq:DoubleW} we have that
  \begin{equation}
    (a,n)(\nu, \kappa)_* = 
    \left(
    a+\nu
     , 
    e^{2a}(n-\kappa e^{2\nu})E
     , 
    \frac{e^{-2a}}{2}(e^{-2\nu}-1)F
     , 
    -\frac{1}{4}(\kappa+e^{2\nu}\kappa - ne^{-2\nu}+n)
    \right).
    \end{equation}
    Similarly, we have
    \begin{equation}
      (\underline{\nu}, \underline{\kappa})_*(\underline{a},\underline{n}) = 
      \left(
      \underline{\nu}+\underline{a}
      , 
      e^{2\underline{\nu}}(e^{2\underline{a}}\underline{n} -\underline{\kappa} )E\
      , 
      \frac{1}{2}(e^{-2\underline{\nu}}-1)F
      , 
      -
      \frac{1}{4}\left(\underline{\kappa}(1+e^{2\underline{\nu}})
      +
      (1-e^{2\underline{\nu}})e^{2\underline{a}}\underline{n}\right)
      \right) .
    \end{equation}
    The dressing action $S^\star\times S\to S^\star$ therefore amounts
    to solve the equation $(a,n)(\nu, \kappa)_* = (\underline{\nu},
    \underline{\kappa})_*(\underline{a},\underline{n})$ for
    $(\underline{\nu},\underline{\kappa})_*$ as a function of
    $a,n,\kappa,\nu$.  From an easy computation it follows that the
    solution is given by
    \begin{equation}
      \left\{
      \begin{array}{ccc}
        \underline{\kappa}&=&\kappa - n \eta(\nu)\\
        \eta(\underline{\nu})&=&e^{-2a} \eta(\nu)
      \end{array}
      \right.
    \end{equation}
    where $\eta$ is the diffeomorphism defined by $ \eta: \mathbb R
    \to ]-1 , \infty[ \colon x\mapsto\eta(x) := e^{-2x}-1$.
    Considering the coordinate system $
    S^\star\hookrightarrow\mathbb R^2:\xi := (\nu,\kappa)_*\mapsto
    (x,y) := (\kappa , \eta(\nu)) $, the local right dressing
    action then reads:
    \begin{equation}
      (x,y)\cdot(a,n) := (x - n y , e^{-2a} y) .
    \end{equation}
    Indeed, the multiplication map
    \begin{equation}
      \label{eq:decomp}
      S^\star\times S\to G
      : 
      (\xi,x)\mapsto \xi\cdot s
    \end{equation}
    is an open embedding. Hence locally one may set:
    \begin{equation}
      s\cdot \xi 
      =
      \xi^s\cdot s^\xi
      \quad\text{ with}\quad 
      s^\xi\in S
      \quad\text{ and}\quad 
      \xi^s\in S^\star .
    \end{equation}
    One then notes that for all $s_1,s_2\in S$ and $\xi\in S^\star$:
    \begin{equation}
      \xi^{s_1 s_2}(s_1s_2)^\xi
      =  
      s_1s_2\xi = s_1\xi^{s_2}x_2^\xi 
      = 
      \left(\xi^{s_2}\right)^{s_1}s_1^{\xi^{s_2}}s_2^\xi
    \end{equation}
    which implies 
    \begin{equation}
      \xi^{s_1s_2} 
      = 
      \left(\xi^{s_2}\right)^{s_1} .
    \end{equation}
    Hence the map $S^\star\times S\to S^\star:(\xi,s)\mapsto \xi^s$
    which given elements $s\in S$ and $\xi\in S^\star$ expresses the
    $S^\star$-component (local) of the product $s\cdot \xi$ in terms
    of the decomposition \eqref{eq:decomp} is a right action of $S$ on
    $S^\star$.  The latter globalizes under the usual matrix
    left-action of the affine group on the plane as
    \begin{equation}
      S\times\mathbb R^2\to\mathbb R^2 
      : 
      (s=(a,n) , v=(x,y))
      \mapsto  
      s.v 
      := 
      v.s^{-1} 
      := 
      \left(\begin{array}{cc}e^{2a}&0\\ n e^{2a}&1\end{array}\right)
        \left(\begin{array}{c}x\\ y\end{array}\right) .
    \end{equation}
    Now we express the group multiplication in $S^\star$ within the
    above coordinate system:
    \begin{equation}
      (x,y).(x',y') 
      := 
      \Phi\left(\Phi^{-1}(x,y).\Phi^{-1}(x',y')\right) 
      = 
      \left( (y'+1)x + x' , (y'+1)y + y' \right) .
    \end{equation}
    The unit consists in the vector origin $(0,0)$ and the inverse
    (which is only local at the level of the entire ambient space
    $\mathbb R^2$) is given by $(x,y)^{-1} = \frac{1}{y+1}\left( -x ,
    -y \right) $.
    It is useful to rewrite the dressing action using musical notation; 
    in this case we consider the coordinate system 
    \begin{equation}
      S^\star\hookrightarrow\mathbb R^2:\xi 
      := 
      (\nu,\kappa)_*\mapsto (x,y) := (\kappa , \eta(\nu)) ,
    \end{equation}
    where $\eta(x) = 1-e^{-2x}$ and the local right dressing action 
    $\xi.(a,n) := (\underline{\kappa},\underline{\nu})$ then reads:
    \begin{equation}
      \label{eq:globaldressing}
      (x,y)\cdot(a,n)
      = 
      (x + ny, e^{-2a}y).
    \end{equation}
    This implies that the dressing action is infinitesimally generated
    by the following fields:
    \begin{equation}
      {}^\flat\widehat{H}_{(x,y)} 
      := 
      \frac{d}{dt}\at{t=0}(x,y)(t,0) 
      = 
      -2y\partial_y
      \quad
      \text{ and}
      \quad
          {}^\flat\widehat{E}_{(x,y)} 
          := 
          \frac{d}{dt}\at{t=0}(x,y)(0,t) 
          = 
          y\partial_x .
    \end{equation}
    The next step consists in computing explicitely the dressing
    generators. Note that there are 3 Poisson structures involved here
    on the image $U $ of $S^\star\hookrightarrow\mathbb R^2$, the dual
    Poisson Lie group structure
    \begin{equation}
      \pi_*
      =
      2y(y+1)\partial_x\wedge \partial_y,
    \end{equation}
    the Poisson structure $\pi_\ell$ induced by the action
    \begin{equation}
      \pi_\ell
      = 
      2y^2 \partial_{x}\wedge \partial_{y}
    \end{equation}
    and the linear one $\pi_{\lie s^\star}$. 
    It is easy to see that
    \begin{equation}
      \pi_*
      = 
      \pi_\ell + \pi_{\lie s^\star}.
    \end{equation}
    Finally, imposing the condition \eqref{eq:DressShift} we obtain
    that the dressing generators with respect to to $\pi_*$ and $\pi_\ell$ we get
    the expressions \eqref{eq:dressingstar} and
    \eqref{eq:dressinglambda}, resp. 
\end{proof}


\begin{thebibliography}{1}

\bibitem {Alekseev2013}
\textsc{Alekseev, A., Meinrenken, E.:}\newblock \emph{Linearization of Poisson
  Lie group structures}.
  \newblock J. Symplectic Geom.  \textbf{14} (2016), 227--267.
  
\bibitem {Bieliavsky2016}
\textsc{Aschieri, P, Bieliavsky, P., Pagani, C., Schenkel, A.:}\newblock \emph{Noncommutative principal bundles through twist deformation}.
\newblock Comm. Math. Phys.  \textbf{352} (2017), 287--344.
  
\bibitem {bayen.et.al:1978a}
\textsc{Bayen, F., Flato, M., Fr{{\o}}nsdal, C., Lichnerowicz, A., Sternheimer,
  D.: }\newblock \emph{Deformation Theory and Quantization}.
\newblock Ann. Phys.  \textbf{111} (1978), 61--151.

\bibitem {Bhaskara1988}
\textsc{Bhaskara, K., Viswanath, K.:}\newblock \emph{Calculus on Poisson
  manifolds}.
\newblock Bull. London Math. Soc.  \textbf{20} (1988), 68--72.

\bibitem{BBM} 
\textsc{Bieliavsky, P., Bonneau, Ph., Maeda, Y.:}
\newblock\emph{Universal deformations formulae, symplectic Lie groups and symmetric spaces}; 
\newblock Pacific J. Math.  \textbf{230} (2007), 41--57.

\bibitem {Bieliavsky2015}
\textsc{Bieliavsky, P., Gayral, V.:}\newblock \emph{Deformation Quantization for Actions of Kählerian Lie Groups}.
\newblock  Mem. Amer. Math. Soc.  \textbf{236} (2015).



\bibitem {thomas1}
\textsc{Bieliavsky, P., Esposito, C., Waldmann, S., Weber, T.: }\newblock \emph{Obstructions for Twist Star Products}.
\newblock Lett. Math. Phys. \textbf{108} (2018), 1341--1350.

\bibitem {chari.pressley:1994a}
\textsc{Chari, V., Pressley, A.: }\newblock \emph{A Guide to Quantum Groups}.
\newblock Cambridge University Press, Cambridge, 1994.

\bibitem {drinfeld:1983}
\textsc{Drinfeld, V.~G.:}
\newblock
\emph{Hamiltonian structures on Lie groups, Lie bialgebras, and the geometric meaning of the classical Yang-Baxter equations}.
\newblock Soviet Math. Dokl.   \textbf{27} (1983), 285--287.

\bibitem {drinfeld:1988a}
\textsc{Drinfeld, V.~G.:}\newblock \emph{Quantum groups}.
\newblock J. Sov. Math.  \textbf{41} (1988), 898--918.

\bibitem {thomas2}
\textsc{D'Andrea, F., Weber, T.: }\newblock \emph{Twist star products and Morita equivalence}.
\newblock C.R. Acad. Sci. Paris, Ser. I \textbf{355} (2017), 1178--1184.


\bibitem {chiara1}
\textsc{Esposito, C.:}\newblock \emph{Poisson Reduction}.
\newblock Trends in Mathematics, Geom. Methods Phys. (2014), 131--142.

\bibitem {jonas}
\textsc{Esposito, C., Schnitzer, J., Waldmann, S.: }\newblock \emph{A Universal Construction of Universal Deformation Formulas,
  Drinfel'd Twists and their Positivity}.
\newblock Pacific J. Math.  \textbf{291} (2017), 319--358.

\bibitem {fedosov}
\textsc{Fedosov, B.:}\newblock \emph{Non Abelian Reduction in deformation quantization}.
\newblock Lett. Math. Phys. \textbf{43} (1998), 137--154.

\bibitem {gerstenhaber}
\textsc{Gerstenhaber, M.: }\newblock \emph{New Universal Deformation Formulas for deformation quantization}.
\newblock Preprint  arXiv:1802.04919.

\bibitem {giaquinto.zhang:1998a}
\textsc{Giaquinto, A., Zhang, J.~J.: }\newblock \emph{Bialgebra actions,
  twists, and universal deformation formulas}.
\newblock J. Pure Appl. Algebra  \textbf{128}.2 (1998), 133--152.

\bibitem {kontsevich}
\textsc{Kontsevich, M.: }\newblock \emph{Deformation quantization of Poisson manifolds}.
\newblock Lett. Math. Phys. \textbf{66}.3 (2003), 157--216.

\bibitem {Kosmann-Schwarzbach2004}
\textsc{Kosmann-Schwarzbach, Y.:}\newblock \emph{Lie bialgebras, Poisson Lie
  groups and dressing transformations}.
\newblock In: \emph{Integrability of Nonlinear Systems}, vol. 638 in
  \emph{Lecture Notes in Physics},   107---173. Springer-Verlag, second.
  edition, 2004.

\bibitem {Lu1990}
\textsc{Lu, J.-H.:}\newblock \emph{{Multiplicative and {A}ffine {P}oisson
  structure on {L}ie groups}}.
\newblock PhD thesis, University of California (Berkeley), 1990.

\bibitem {Lu1990b}
\textsc{Lu, J.-H., Weinstein, A.:}\newblock \emph{{Poisson Lie groups, dressing transformations, and Bruhat decompositions}}.
\newblock J. Differential Geom. \textbf{31}.2 (1990), 501--526.

\bibitem {Lu1991}
\textsc{Lu, J.-H.:}\newblock \emph{Momentum mappings and reduction of Poisson actions}.
\newblock In: \emph{Symplectic geometry,
groupoids, and integrable systems},  Math. Sci. Res. Inst. Publ., Berkeley, California, 1991.

\bibitem {lu3}
\textsc{Lu, J.-H.:}\newblock \emph{Moment Maps at the Quantum level}.
\newblock Comm. Math. Phys. \textbf{157}.2 (1993), 389--404.

\bibitem {Majid}
\textsc{Majid, S.:}\newblock \emph{Foundations of Quantum Group Theory}.
\newblock Cambridge University Press (1995).



\end{thebibliography}
\end{document}